\documentclass[11pt]{amsart}
\usepackage{amssymb}
\usepackage{fullpage}
\usepackage{color}
\usepackage[T1]{fontenc}    
\usepackage[utf8]{inputenc}
\usepackage{enumerate}

\newtheorem{Theorem}{Theorem}[section]
\newtheorem{Proposition}[Theorem]{Proposition}
\newtheorem{Lemma}[Theorem]{Lemma}

\newcommand{\RR}{\mathbb{R}}

\providecommand{\norm}[1]{\left\lVert#1\right\rVert}
\newcommand{\qtq}[1]{\quad \text{#1}\quad }

\begin{document}

\title{Well-posedness and exponential decay estimates for a Korteweg--de Vries--Burgers equation with time-delay}

\author{Vilmos Komornik}
\address{
Département de mathématique,
Université de Strasbourg,
7 rue René Descartes,
67084 Strasbourg Cedex, France}
\email{komornik@math.unistra.fr}

\author{Cristina Pignotti}
\address{Dipartimento di Ingegneria e Scienze dell'Informazione e Matematica,
Università di L'Aquila,
Via Vetoio, Loc. Coppito, 67010 L'Aquila, Italy}
\email{pignotti@univaq.it}

\date{Version 2019-06-02}

\numberwithin{equation}{section}

\begin{abstract}
We consider the KdV--Burgers equation and its linear version in presence of a delay feedback. We prove well-posedness of the models and exponential decay estimates under appropriate conditions on the damping coefficients. Our arguments rely on a Lyapunov functional approach  combined with a step by step procedure and semigroup theory.
\end{abstract}
\maketitle

\section{Introduction}
\label{s1}

The aim of this paper is to investigate   stability properties of the Cauchy problem

\begin{equation}\label{11}
\begin{cases}
u_t(x,t)+u_{xxx}(x,t)-u_{xx}(x,t) +\lambda_0 u(x,t)\\
\hspace{4cm}+\lambda u(x,t-\tau )+u(x,t) u_x(x,t)=0
&\qtq{in}\RR\times (0,\infty),\\
u(x, s)=u_0(x, s)&\qtq{in}\RR\times [-\tau, 0]
\end{cases}
\end{equation}
and its linear version

\begin{equation}\label{12}
\begin{cases}
u_t(x,t)+u_{xxx}(x,t)-u_{xx}(x,t) +\lambda_0 u(x,t) +\lambda u(x,t-\tau )=0
&\qtq{in}\RR\times (0,\infty),\\
u(x,s)=u_0(x,s)&\qtq{in}\RR\times [-\tau, 0].
\end{cases}
\end{equation}
Here the constant $\tau >0$ is the time delay and the coefficients $\lambda_0(x), \lambda(x)$ belong to $L^\infty(\RR).$

The Korteweg--de Vries--Burgers equation
\begin{equation}\label{13}
u_t+u_{xxx}-u_{xx} +uu_x=0  \qtq{in} \RR\times (0,\infty)
\end{equation}
models the unidirectional propagation of planar waves. The function $u=u(x,t)$ represents the amplitude of the wave at position $x$ and at time $t.$
In \cite{Amick} the authors proved that the $L^2-$norm of solutions to \eqref{13} tends to zero as $t\rightarrow\infty$ in a polynomial way, namely
\begin{equation*}
\norm{u(\cdot, t)}_{L^2(\RR)}\le C t^{-\frac 1 2}\qtq{for all} t>0,
\end{equation*}
with a positive constant $C.$
In \cite{CCKR} a damped KdV--Burgers equation is considered, namely

\begin{equation}\label{14}
\begin{cases}
u_t(x,t)+u_{xxx}(x,t)-u_{xx}(x,t) +\lambda_0 u(x,t) +u(x,t) u_x(x,t)=0
&\qtq{in}\RR\times (0,\infty),\\
u(x, 0)=u_0(x)&\qtq{in}\RR
\end{cases}
\end{equation}
together with its linear version, i.e., without the term $u\,u_x.$
The authors investigated the well-posedness and exponential stability for an indefinite damping $\lambda_0(x),$ giving exponential decay estimates on the $L^2-$norm of solutions to \eqref{14} under appropriate conditions on the damping coefficient $\lambda_0.$

The damped KdV equation
\begin{equation*}
u_t+bu_x+u_{xxx}+u u_x + a u=0
\qtq{in}\RR\times (0,\infty)
\end{equation*}
is instead studied in \cite{CCFN, LP}.
Concerning the KdV equation in a finite interval with localized damping, exponential decay estimates have been obtained in \cite{MVZ,P}.
Periodic conditions have been considered in \cite{Komornik, KRZ} while more general nonlinearities have been considered in \cite{RZ}.

In order to take into account the physical meaning of the models, it is natural to include delay effects. It is by now well-known from pioneer papers of Datko \cite{Datko}, Datko et al. \cite{DLP},  that an arbitrarily small time delay may gives instability phenomena in models which are uniformly asymptotically stable in absence of delay.
Nevertheless, appropriate choices of the time delay can restitute stability (cf. \cite{Gugat}) as well as appropriate feedback laws (cf. \cite{XYL, NPSicon06, NP15, LasWeb2016, NP18, KomPig2018}).  Then, our aim here is to furnish sufficient conditions on the coefficients $\lambda, \lambda_0$ in order to have well-posedness of the models (\ref{12}) and (\ref{11}) and exponential decay estimates.
We emphasize the fact that the results here obtained could not  be deduced from the general approaches of \cite{NP15, NP18} or \cite{KomPig2018}. Indeed, the methods there proposed would require a smallness assumption on the $L^\infty-$ norm of the delay feedback coefficient $\lambda_0.$
A KdV model in a finite interval with time delay in the boundary condition has been recently studied in \cite{BCV}. Concerning the KdV--Burgers equation in a bounded interval, a model with input delay and constant coefficient of the undelayed damping has been recently analyzed in \cite{KF}.

Note that under the assumption
\begin{equation}\label{15}
\lambda_0 (x) \ge \alpha_0 \qtq{for a.e.} x\in \RR
\end{equation}
with some positive constant $\alpha_0,$
if the coefficient of the delay term $\lambda$ satisfies the estimate
$\norm{\lambda}_\infty <\alpha_0,$  then we could easily obtain exponential decay estimates. Indeed, in such a case  the delay effect is compensated  by the undelayed  damping term (cf. \cite{NPSicon06, XYL}).

However, we will deal here with a more general setting.
First, for the sake of clearness, we restrict ouselves to the case of $\lambda_0$ bounded from below by a positive constant but
 it may be $\vert \lambda(x)\vert \ge \lambda_0 (x)$ in some part of the domain. Then, we extend our results to the case in which  the coefficient of the undelayed feedback $\lambda_0$ is also indefinite.

The paper is organized as follows. In section \ref{s2} we analyze well-posedness and exponential decay of the problem \eqref{12} under the assumption \eqref{15} while in section \ref{s3} we will focus on the nonlinear model \eqref{11} under the same assumption on the coefficient $\lambda_0$ of the undelayed feedback. Finally, in section \ref{s4} we generalize the results of previous sections by removing assumption \eqref{15}.

\section{The linearized KdV--Burgers equation\label{s2}}

First we analyze the linear model \eqref{12}. We prove the well-posedness via a step by step procedure. Then, under suitable conditions on the coefficients $\lambda$ and $\lambda_0,$ we deduce an exponential stability estimate.

\subsection{Well-posedness of the linear model\label{ss21}}

First we look at the problem
\begin{equation}\label{21}
\begin{cases}
u_t(x,t)+u_{xxx}(x,t)-u_{xx}(x,t) +\lambda_0 u(x,t)=0&\qtq{in}\RR\times (0,\infty),\\
u(x, 0)=u_0(x, 0)&\qtq{in}\RR.
\end{cases}
\end{equation}
The following well-posedness result is proved
in \cite{CCKR}.

\begin{Proposition}\label{p21}
If $\lambda_0\in L^\infty (\RR ),$ then the operator
$A_{\lambda_0}$ defined by the formula $A_{\lambda_0} u:=-u_{xxx}+u_{xx}-\lambda_0 u$ on ${\mathcal D}(A_{\lambda_0} ):=H^3(\RR)$ generates a strongly continuos semigroup in the Hilbert space $H:=L^2(\RR ).$
\end{Proposition}

Now, using an iterative procedure (see e.g. \cite{NP15}) and standard semigroup arguments  (see e.g. \cite{Pazy}), we
can prove a well-posedness result for the problem \eqref{12}.

\begin{Theorem}\label{t22}
If $\lambda_0,\lambda\in L^\infty (\RR )$ and $u_0\in C([-\tau ,0]; H)$,  then there exists a unique solution $u\in C([-\tau, +\infty); H)$ of the problem \eqref{12}.
\end{Theorem}

\begin{proof}
First, we argue on the interval $[0,\tau].$
Then \eqref{12}  may be regarded as an inhomogeneous Cauchy problem of the form
\begin{equation}\label{22}
\begin{cases}
u_t(t)-A_{\lambda_0}u(t)=g_0(t)&\qtq{in}(0,\tau ),\\
u(0)=u_0,
\end{cases}
\end{equation}
where $g_0(t)=-\lambda u_0(t-\tau),$ for $t\in [0,\tau].$
This problem admits a unique solution $u(\cdot) \in C([0,\tau), H).$
Now, we consider $t\in [\tau, 2\tau ].$ Then, problem \eqref{12} can be rewritten as
\begin{equation}\label{23}
\begin{cases}
u_t(t)-A_{\lambda_0}u(t)=g_1(t)&\qtq{in}(\tau,2\tau ),\\
u(\tau )=u({\tau}_{-}),
\end{cases}
\end{equation}
with $g_1(t)=-\lambda u(t-\tau).$ Observe that we know $u(t)$ for $t\in [0,\tau]$ from the first step; so $g_1(t)$ can be considered as a known function for $t\in [\tau , 2\tau].$ Therefore, we deduce the existence of a solution $u(\cdot )\in C ([0, 2\tau ], H).$ By iterating this procedure we get a solution $u\in C([0, \infty ), H).$
\end{proof}

\subsection{Asymptotic stability of the linear model}\label{ss22}

Let us define the Lyapunov functionals
\begin{equation}\label{24}
E(t):= E(u(t))=\frac 12\int_\RR u^2 (x,t) dx
\end{equation}
and, for $\lambda\in L^\infty (\RR ),$
\begin{equation}\label{25}
{\mathcal E}(t):= {\mathcal E}(u(t))=\frac 12\int_\RR u^2 (x,t) dx+\frac 12\int_{t-\tau }^t\int_\RR e^{-(t-s)}\vert \lambda (x)\vert u^2 (x,s)\ dx\  ds.
\end{equation}

Setting
\begin{equation}\label{26}
c_p:=\Big (1 - \frac 1 {2p}\Big ) \Big (\frac 2 p \Big )^{\frac 1 {2p-1}} \qtq{for} 1\le p<\infty,
\end{equation}
we can prove the following exponential stability result.

\begin{Theorem}\label{t23}
Assume that $\lambda , \lambda_0\in L^\infty (\RR)$ and $\lambda_0$ satisfies \eqref{15}.
If  there exist a positive constant $\alpha$ and a function $\beta\in L^p(\RR),$ for some $1\le p<\infty$  such that the function $\lambda$ satisfies
\begin{equation}\label{27}
\frac {e^\tau +1} 2\vert \lambda (x)\vert \le \alpha + \beta(x)\qtq{for a.e.}x\in\RR
\end{equation}
with
\begin{equation}\label{28}
0\le\alpha< \alpha_0 \qtq{and}
\norm{\beta}_p<\Big(\frac{\alpha_0-\alpha }{c_p}\Big)^{1-\frac 1 {2p}},
\end{equation}
where $c_p$ is defined in $(\ref{26}),$ then the problem \eqref{12} is exponentially stable.
In particular, the solutions $u$ of \eqref{12} satisfy the inequalities
\begin{equation}\label{29}
{\mathcal E}(t) \le C(u_0) e^{-\gamma t}
\end{equation}
where
\begin{equation}\label{210}
\gamma = \min \left\{2 \Big (
\alpha_0 -\alpha -\frac {2p-1}{2p} \Big (\frac 2 p \Big )^{\frac 1 {2p-1}}
\Vert \beta \Vert_p^{\frac {2p}{2p-1}}
\Big ),1\right\}
\end{equation}
and
\begin{equation}\label{211}
C(u_0)=\frac 12\Vert u(0)\Vert^2_2 +\int_{-\tau}^0 e^s\vert \lambda\vert\cdot \Vert u(s)\Vert_2^2 ds.
\end{equation}
\end{Theorem}

\begin{proof}
For the computations we consider $u_0\in H^3,$ then $u\in H^3$ (see \cite[Th. 4.7]{CCKR}).
We then extend the result to every solution in $H$ by density.
By differentiating ${\mathcal E}(t)$ we obtain
\begin{multline*}
\frac {d{\mathcal E}} {dt} (t)
=\int_\RR u(t)(u_{xx}(t)-\lambda_0 u(t)-\lambda u(t-\tau )) dx
+\frac 12\int_\RR \vert \lambda\vert u^2(t) dx\\
- \frac 12e^{-\tau}\int_\RR \vert \lambda\vert u^2(t-\tau ) dx
-\frac 12\int_{t-\tau }^t\int_\RR e^{-(t-s)}\vert \lambda (x)\vert u^2 (x,s)dx \, ds,
\end{multline*}
where we used the equation and the fact that
\begin{equation*}
\int_\RR u  u_{xxx}\, dx =0
\qtq{for all}u\in H^3.
\end{equation*}
Then, integrating by parts, using the Young inequality and recalling \eqref{15} and \eqref{27}, we get
\begin{align*}
\frac {d{\mathcal E}} {dt} (t)
&\le -\int_\RR u_x^2(t) dx -\alpha_0 \int_\RR u^2(t) dx +\frac {e^\tau +1} 2\int_\RR \vert \lambda(x)\vert  u^2 (t) dx\\
&\qquad\qquad -\frac 12\int_{t-\tau }^t\int_\RR e^{-(t-s)}\vert \lambda (x)\vert u^2 (x,s)dx \, ds\\
&\le-\int_\RR u_x^2(t) dx - (\alpha_0 -\alpha) \int_\RR u^2(t) dx +
\int_\RR \beta(x) u^2 (t) dx\\
&\qquad\qquad -\frac 12\int_{t-\tau }^t\int_\RR e^{-(t-s)}\vert \lambda (x)\vert u^2 (x,s)\ dx \ ds.
\end{align*}
Using the Hölder inequality, hence we deduce that
\begin{equation}\label{212}
\frac {d{\mathcal E}} {dt} (t)\le-\Vert u_x(t)\Vert_2^2 -(\alpha_0 -\alpha )\Vert u(t)\Vert_2^2
+\Vert \beta\Vert_p \Vert u\Vert_{2q}^2 -\frac 12\int_{t-\tau }^t\int_\RR e^{-(t-s)}\vert \lambda (x)\vert u^2 (x,s)dx \, ds,
\end{equation}
where $q=\frac p {p-1}.$
Now observe that
\begin{equation}\label{213}
\Vert u\Vert_{2q}^2=\left (\int_\RR (u(t))^{2q} dx\right )^{\frac 1q}= \left (\int_\RR u^2(t) (u(t))^{\frac 2 {p-1}} dx\right )^{\frac 1 q}
\le \Vert u\Vert_2^{\frac 2 q}\Vert u\Vert_\infty^{\frac 2 {q(p-1)}}= \Vert u\Vert_2^{\frac 2 q}\Vert u\Vert_\infty^{\frac 2 p}.
\end{equation}

Then using \eqref{213} in \eqref{212} we  obtain
\begin{multline}\label{214}
\frac {d{\mathcal E}} {dt} (t)\le-\Vert u_x(t)\Vert_2^2 -(\alpha_0 -\alpha )\Vert u(t)\Vert_2^2+\Vert \beta\Vert_p \Vert u\Vert_2^{\frac 2 q} \Vert u(t)\Vert_{\infty}^{\frac 2 p}  \\
-\frac 12\int_{t-\tau }^t\int_\RR e^{-(t-s)}\vert \lambda (x)\vert u^2 (x,s)dx \, ds.
\end{multline}
Therefore, observing that (see \cite{CCKR})
\begin{equation}\label{215}
\Vert v\Vert_\infty^2\le 2\Vert v\Vert_2\Vert v_x\Vert_2
\end{equation}
for all $v\in H^1(\RR)$, and using the Young inequality, from \eqref{214} we deduce for every fixed $\delta >0$ the following inequalities:
\begin{equation*}
\begin{split}
\frac {d{\mathcal E}} {dt} (t)
&\le
-\Vert u_x(t)\Vert_2^2 -(\alpha_0 -\alpha)\Vert u(t)\Vert_2^2+2^{1/p}
\Vert \beta \Vert_p\Vert u\Vert_2^{\frac {2p-1} p} \Vert u_x\Vert_2^{\frac 1 p}\\
&\qquad\qquad -\frac 12\int_{t-\tau }^t\int_\RR e^{-(t-s)}\vert \lambda (x)\vert u^2 (x,s)dx \, ds\\
&\le -\Vert u_x(t)\Vert_2^2 -(\alpha_0 -\alpha)\Vert u(t)\Vert_2^2
+\Big ( \frac 1 {\delta } \Vert \beta \Vert_p\Vert u\Vert_2^{\frac {2p-1} p}  \Big ) \left ( \delta 2^{\frac 1 p}\Vert u_x\Vert_2^{\frac 1 p}      \right )
\\
&\qquad\qquad -\frac 12\int_{t-\tau }^t\int_\RR e^{-(t-s)}\vert \lambda (x)\vert u^2 (x,s)dx \, ds\\
&\le -\Vert u_x(t)\Vert_2^2 -(\alpha_0 -\alpha)\Vert u(t)\Vert_2^2
+\frac {
\Big ( \frac 1 {\delta } \Vert \beta \Vert_p\Vert u\Vert_2^{\frac {2p-1} p}  \Big )^\frac {2p}{2p-1}}{\frac {2p}{2p-1}} +
 \frac {
 \left ( \delta 2^{\frac 1 p}\Vert u_x\Vert_2^{\frac 1 p}      \right )^{2p}}{2p}\\
&\qquad\qquad -\frac 12\int_{t-\tau }^t\int_\RR e^{-(t-s)}\vert \lambda (x)\vert u^2 (x,s)dx \, ds.
\end{split}
\end{equation*}
Choosing $\delta$ such that $4\delta ^{2p}=2p,$ this yields
\begin{multline*}
\quad \frac {d{\mathcal E}} {dt} (t)\le - \Big (
\alpha_0 -\alpha -\frac {2p-1}{2p} \Big (\frac 2 p \Big )^{\frac 1 {2p-1}}
\Vert \beta \Vert_p^{\frac {2p}{2p-1}}
\Big ) \Vert u(t)\Vert_2^2\\
-\frac 12\int_{t-\tau }^t\int_\RR e^{-(t-s)}\vert \lambda (x)\vert u^2 (x,s)dx \, ds.\quad
\end{multline*}
Thus, under the assumption \eqref{28} we have
\begin{equation*}
\frac {d{\mathcal E}} {dt} (t)\le -  {\gamma} {\mathcal E}(t)
\end{equation*}
with
$\gamma$ as in \eqref{210}.
Now the estimate \eqref{29} follows from Gronwall's Lemma
with  $C(u_0) ={\mathcal E}(0).$
\end{proof}

\section{The nonlinear model
\label{s3}}

In order to prove the well-posedness of the nonlinear model \eqref{11} first we  consider the corresponding linear inhomogeneous initial value problem
\begin{equation}\label{31}
\begin{cases}
u_t(x,t)+u_{xxx}(x,t)-u_{xx}(x,t) +\lambda_0 u(x,t)+\lambda u(x,t-\tau)=f(x,t) &\qtq{in}\RR\times (0,T),\\
u(x,s)=u_0(x,s) &\qtq{in}\RR\times [-\tau, 0]
\end{cases}
\end{equation}
for some $T>0.$
Setting
\begin{equation*}
A_{\lambda_0}:= -\partial_x^3+\partial_x^2-\lambda_0 I,\quad {\mathcal{D}}(A_{\lambda_0})=H^3(\RR),
\end{equation*}
we can rewrite \eqref{31} in the form
\begin{equation}\label{32}
\begin{cases}
u_t(x,t)+\lambda u(x,t-\tau )=A_{\lambda_0} u(x,t) +f(x,t)&\qtq{in}\RR\times (0,T),\\
u(x,s)=u_0(x,s)&\qtq{in}\RR\times [-\tau, 0].
\end{cases}
\end{equation}

We know that $A_{\lambda_0}$ generates a strongly continuous semigroup of contractions in $L^2(\RR )$ (see \cite{CCKR}). Then, for any data $u_0\in C([-\tau, 0], H)$ and $f\in L^1(0,T;L^2(\RR)),$ the problem \eqref{32} has a unique mild solution $u\in C([-\tau, T];L^2(\RR)),$ satisfying the representation formula

\begin{equation}\label{33}
u(t)=S(t)u_0(0)-\int_0^t S(t-s)\lambda u(s-\tau )\, ds +\int_0^t S(t-s) f(s) \, ds,\quad t\in [0,T].
\end{equation}
One can show that the mild solution depends continuously on the initial data.

\begin{Proposition}\label{p31}
If $u_0\in C([-\tau, 0], H)$ and $f\in L^1(0,T;L^2(\RR)),$ then the solution of  \eqref{32}
satisfies the following estimate:
\begin{equation}\label{34}
\Vert u(t)\Vert_
{C([0,T];L^2(\RR))}\le
e^{\Vert\lambda\Vert_\infty T}
\left ( \Vert u(0)\Vert_{L^2(\RR)}+\Vert f\Vert_{L^1(0,T;L^2(\RR))}+
\Vert \lambda\Vert_\infty\int_{-\tau }^0\Vert u(s)\Vert_{L^2(\RR)} ds \right ).
\end{equation}
\end{Proposition}

\begin{proof}
It follows from the representation formula \eqref{33} that
\begin{align*}
\Vert u(t)\Vert_{L^2(\RR)}
&\le \Vert u(0)\Vert_{L^2(\RR)}
+\Vert \lambda\Vert_\infty\int_0^t\Vert u(s-\tau )\Vert_{L^2(\RR)}\, ds+\Vert f\Vert_{L^1(0,T;L^2(\RR))}\\
&\le \Vert u(0)\Vert_{L^2(\RR)}+\Vert f\Vert_{L^1(0,T;L^2(\RR))}+\Vert \lambda\Vert_\infty\int_{-\tau}^0\Vert u(s)\Vert_{L^2(\RR)}ds\\
&\qquad\qquad+
\Vert \lambda\Vert_\infty\int_0^t\Vert u(s)\Vert_{L^2(\RR)}\, ds .
\end{align*}
Then Gronwall's lemma implies \eqref{34}.
\end{proof}

Actually, the solution of \eqref{32} has an additional regularity.
Let us introduce the Banach space
\begin{equation*}
{\mathcal{B}}_T
:= C([0,T];L^2(\RR ))\cap L^2(0,T;H^1(\RR))
\end{equation*}
with the norm
\begin{equation*}
\Vert u\Vert_{\mathcal{B}_T}=\Vert u\Vert_{ C([0,T];L^2(\RR))}+\Vert \partial_xu\Vert_{L^2(0,T;L^2(\RR))}.
\end{equation*}
The following proposition holds.

\begin{Proposition}\label{p32}
If $u_0\in C([-\tau, 0], H)$ and $f\in L^1(0,T;L^2(\RR)),$ then the solution of  \eqref{32} belongs to ${\mathcal B}_T$ and
satisfies the estimate
\begin{equation}\label{35}
\Vert u\Vert_{{\mathcal B}_T}\le C_T \left \{ \Vert u(0)\Vert_{L^2(\RR)}+\Vert f\Vert_{L^1(0,T;L^2(\RR))}+
\left ( \Vert \lambda\Vert_\infty \tau^{1/2} + \Vert \lambda\Vert_\infty^{1/2} \right )\Vert u\Vert_{L^2(-\tau, 0; L^2(\RR))}  \right \}
\end{equation}
with
\begin{equation}\label{36}
C_T=\sqrt{\frac 3 2} \left (1+e^{2\Vert \lambda\Vert_\infty T}\right )^{1/2}
e^{(\Vert \lambda\Vert_\infty +\Vert \lambda_0\Vert_\infty) T}.
\end{equation}
Moreover, the following identity holds for all $t\in [0,T]$:
\begin{multline}\label{37}
\frac 12 \Vert u(t)\Vert_{L^2(\RR)}^2+\int_0^t\Vert u_x\Vert^2_{L^2(\RR)} ds+\int_0^t\int_\RR \lambda_0 u^2(x,s) \,dx\, ds
+\int_0^t\int_\RR\lambda u(x, s-\tau )u(x, s) \,dx\, ds \\
=\frac 12 \Vert u(0)\Vert_{L^2(\RR)}^2+\int_0^t\int_\RR f(x,s) u(x,s)\, dx\,ds.
\end{multline}
\end{Proposition}

\begin{proof}
Multiplying the equation by $u$ and integrating by parts we obtain \eqref{37}.
Using \eqref{34} hence we infer that
\begin{multline*}
\Vert u(t)\Vert_{L^2(\RR)}^2
+2\int_0^t \Vert u_x\Vert_{L^2(\RR)}^2 ds
\le
\Vert u(0)\Vert_{L^2(\RR)}^2\\
+2\Vert f\Vert_{L^1(0,T;L^2(\RR))} e^{\Vert \lambda\Vert_\infty T}\left ( \Vert u(0)\Vert_{L^2(\RR)}+\Vert f\Vert_{L^1(0,T;L^2(\RR))}+
\Vert \lambda\Vert_\infty\int_{-\tau }^0\Vert u(s)\Vert_{L^2(\RR)} ds \right ) \\
+2\Vert\lambda_0\Vert_\infty \int_0^t\Vert u(s)\Vert_{L^2(\RR)}^2 ds +\Vert\lambda\Vert_\infty\int_0^t\Vert u(s-\tau )\Vert_{L^2(\RR)}^2 ds +\Vert\lambda\Vert_\infty\int_0^t\Vert u(s)\Vert_{L^2(\RR)}^2 ds.
\end{multline*}
Thus we have
\begin{multline}\label{38}
\Vert u(t)\Vert_{L^2(\RR)}^2 +2\int_0^t \Vert u_x\Vert_{L^2(\RR)}^2 ds \le \Vert u(0)\Vert_{L^2(\RR)}^2+\Vert f\Vert_{L^1(0,T;L^2(\RR))}^2\\
+
e^{2\Vert \lambda\Vert_\infty T}\left ( \Vert u(0)\Vert_{L^2(\RR)}+\Vert f\Vert_{L^1(0,T;L^2(\RR))}+
\Vert \lambda\Vert_\infty \sqrt{\tau} \Vert u\Vert_{L^2(-\tau, 0; L^2(\RR))}  \right )^2 \\
+\Vert \lambda\Vert_\infty
\Vert u\Vert_{L^2(-\tau, 0; L^2(\RR))}^2
+ 2(\Vert \lambda\Vert_\infty +\Vert \lambda_0\Vert_\infty)
\int_0^t\Vert u\Vert^2_{L^2(\RR)} ds,
\end{multline}
where we have used once again the inequality
\begin{equation*}
\int_0^t\Vert u(s-\tau )\Vert_{L^2(\RR)}^2 ds\le \int_{-\tau}^0\Vert u(s)\Vert_{L^2(\RR)}^2 ds +\int_0^t\Vert u(s)\Vert_{L^2(\RR)}^2.
\end{equation*}
From \eqref{38} we have
\begin{multline*}
\Vert u(t)\Vert_{L^2(\RR)}^2 +2\int_0^t \Vert u_x\Vert_{L^2(\RR)}^2 ds \\
\le \Big (1+e^{2\Vert \lambda\Vert_\infty T}\Big )
\Big \{ \Vert u(0)\Vert_{L^2(\RR)}+\Vert f\Vert_{L^1(0,T;L^2(\RR))}+
 \Big ( \Vert \lambda\Vert_\infty \tau^{1/2} + \Vert \lambda\Vert_\infty^{1/2}  \Big )
 \Vert u\Vert_{L^2(-\tau, 0; L^2(\RR))}  \Big \}^2\\
+ 2( \Vert \lambda\Vert_\infty +\Vert \lambda_0\Vert_\infty)
\int_0^t\Vert u\Vert^2_{L^2(\RR)} ds
,
\end{multline*}
and then from Gronwall's Lemma we get
\begin{multline*}
\Vert u(t)\Vert_{L^2(\RR)}^2 +2\int_0^t \Vert u_x\Vert_{L^2(\RR)}^2 ds
\le \left (1+e^{2\Vert \lambda\Vert_\infty T}\right )
 e^{2(\Vert \lambda\Vert_\infty +\Vert \lambda_0\Vert_\infty) T}\times\\
\times
\Big \{ \Vert u(0)\Vert_{L^2(\RR)}+\Vert f\Vert_{L^1(0,T;L^2(\RR))}+
 \Big ( \Vert \lambda\Vert_\infty \tau^{1/2} + \Vert \lambda\Vert_\infty^{1/2}  \Big )
 \Vert u\Vert_{L^2(-\tau, 0; L^2(\RR))}  \Big \}^2.
\end{multline*}
Therefore
\begin{multline*}
\Vert u\Vert_{{\mathcal B}_T}^2\le
\frac 3 2 \left (1+e^{2\Vert \lambda\Vert_\infty T}\right )
e^{2(\Vert \lambda\Vert_\infty +\Vert \lambda_0\Vert_\infty) T}
\times\\
\times
\Big \{ \Vert u(0)\Vert_{L^2(\RR)}+\Vert f\Vert_{L^1(0,T;L^2(\RR))}+
 \Big ( \Vert \lambda\Vert_\infty \tau^{1/2} + \Vert \lambda\Vert_\infty^{1/2}  \Big )
 \Vert u\Vert_{L^2(-\tau, 0; L^2(\RR))}  \Big \}^2,
\end{multline*}
and so
\begin{multline*}
\Vert u\Vert_{{\mathcal B}_T}\le
\sqrt{\frac 3 2} \left (1+e^{2\Vert \lambda\Vert_\infty T}\right )^{1/2}
e^{(\Vert \lambda\Vert_\infty +\Vert \lambda_0\Vert_\infty) T}
\times\\
\times
\Big \{ \Vert u(0)\Vert_{L^2(\RR)}+\Vert f\Vert_{L^1(0,T;L^2(\RR))}+
 \Big ( \Vert \lambda\Vert_\infty \tau^{1/2} + \Vert \lambda\Vert_\infty^{1/2}  \Big )
 \Vert u\Vert_{L^2(-\tau, 0; L^2(\RR))}  \Big \}.
\end{multline*}
Thus we arrive at
\begin{equation*}
\Vert u\Vert_{{\mathcal B}_T}\le C_T \Big \{ \Vert u(0)\Vert_{L^2(\RR)}+\Vert f\Vert_{L^1(0,T;L^2(\RR))}+
 \Big ( \Vert \lambda\Vert_\infty \tau^{1/2} + \Vert \lambda\Vert_\infty^{1/2}  \Big )
 \Vert u\Vert_{L^2(-\tau, 0; L^2(\RR))}  \Big \}
\end{equation*}
with $C_T$ as in the statement.
\end{proof}

Now we consider the nonlinear model \eqref{11} with $u_0\in C([-\tau , 0]; L^2(\RR)).$
By a mild solution of \eqref{11} we mean a function $u\in {\mathcal B}_T,$ $T>0,$ which satisfies
\begin{equation*}
u(t)=S(t) u_0(0) -\int_0^t S(t-s)\lambda u(s-\tau )\, ds -\int_0^t S(t-s) u(s)\partial_x u(s) \, ds,\quad t\in [0,T].
\end{equation*}
By a global mild solution of \eqref{11} we mean a function $u:[0,\infty)\rightarrow H^1(\RR)$
whose restriction to every bounded interval $[0,T]$ is a mild solution of \eqref{11}.
We have the following well-posedness result.

\begin{Theorem}\label{t33}
Let $\lambda_0, \lambda\in L^\infty (\RR)$  satisfying \eqref{15}, \eqref{27} and \eqref{28}.
Then, for every $u_0\in C([-\tau,0]; L^2(\RR))$ the problem \eqref{11} admits a unique global mild solution.
Moreover the following identity holds for all $t\ge 0$:
\begin{multline}\label{39}
\frac 12 \Vert u(t)\Vert_{L^2(\RR)}^2+\int_0^t\Vert u_x\Vert^2_{L^2(\RR)} ds+\int_0^t\int_\RR \lambda_0 u^2(x,s) \,dx\, ds\\
+\int_0^t\int_\RR\lambda u(x, s-\tau )u(x, s) \,dx\, ds=
\frac 12 \Vert u(0)\Vert_{L^2(\RR)}^2.
\end{multline}
\end{Theorem}
For the proof we need the following lemma (see \cite[Proposition 4.1]{l34}):

\begin{Lemma}\label{l34}
If $u\in L^2(0,T;H^1(\RR)),$ then $uu_x\in L^1(0,T;L^2(\RR)).$ Moreover, if $u, v\in \mathcal{B}_T,$ then
\begin{equation*}
\Vert uu_x-vv_x\Vert_{L^1(0,T;L^2(\RR))}\le\sqrt{2}T^{1/4}\left ( \Vert u\Vert_{\mathcal{B}_T} + \Vert v\Vert_{\mathcal{B}_T}\right ) \Vert u-v\Vert_{\mathcal{B}_T}.
\end{equation*}
\end{Lemma}
Applying a fixed point argument, as in \cite{Pazoto}, we get a local well-posedness result.

\begin{Proposition}\label{p35}
If $u_0\in C([-\tau, 0]; L^2(\RR)),$ then  the problem \eqref{11} has a unique mild solution on $[0,T],$ for a sufficiently small $T>0$.
Moreover, the solution satisfies \eqref{39} for all $t\in [0,T].$
\end{Proposition}

\begin{proof}
By Proposition \ref{p32}, the solution $u$ of  \eqref{31} satisfies the estimate \eqref{35} with the constant $C_T$  defined in \eqref{36}.
Note that $C_T$ non-decreasing in $T.$

Let $u_0 \in C([-\tau, 0]; L^2(\RR))$ be given. In order to prove the existence of a solution of \eqref{11} we introduce a map $M$ defined by
\begin{equation*}
(Mu)(t)=S(t)u_0(0)  -\int_0^t S(t-s)\lambda u_0(s-\tau )\, ds -\int_0^t S(t-s) u(s)u_x(s) \, ds,\quad t\in [0,\tau ],
\end{equation*}
in the space ${\mathcal B}_T$
with the  natural norm.
We will prove that $M$ has a fixed point in some ball $B_R(0)$ of ${\mathcal B}_T.$

We claim that there exists a $K>0$ such that
\begin{equation*}
\Vert Mu -Mv\Vert_{\mathcal{B}_T}\le KT^{1/4}\left (
\Vert u\Vert_{\mathcal{B}_T}+\Vert v\Vert_{\mathcal{B}_T}
\right ) \Vert u-v\Vert_{\mathcal{B}_T}
\end{equation*}
for all $u, v\in {\mathcal{B}_T}$ and for all $0<T\le \tau$.

According to previous observations we have
\begin{equation*}
\Vert Mu -Mv\Vert_{\mathcal{B}_T}\le C\Vert uu_x-vv_x\Vert_{L(0,T;L^2(\RR))}.
\end{equation*}
Therefore, applying the triangle and  Hölder inequalities, we deduce that
\begin{multline}\label{310}
\Vert Mu -Mv\Vert_{\mathcal{B}_T}\le
C\left (
\Vert u-v\Vert_{L^2(0,T;L^\infty(\RR))} \Vert u\Vert_{L^2(0,T;H^1(\RR))}
\right.\\
\left.+\Vert v\Vert_{L^2(0,T;L^\infty(\RR))}\Vert u-v\Vert_{L^2(0,T;H^1(\RR))}
\right ).
\end{multline}
Now from \eqref{215} we have
\begin{equation}\label{311}
\Vert u\Vert_{L^2(0,T;L^\infty(\RR))}\le CT^{1/4}\Vert u\Vert^{1/2}_{L^\infty(0,T;L^2(\RR))}
\Vert u\Vert^{1/2}_{L^2(0,T;H^1(\RR))}.
\end{equation}
From \eqref{311} we deduce in particular the inequality
\begin{equation}\label{312}
\Vert u\Vert_{L^2(0,T;L^\infty(\RR))}\le CT^{1/4}\Vert u\Vert^{1/2}_{ \mathcal{B}_T}
\Vert u\Vert^{1/2}_{ \mathcal{B}_T}= CT^{1/4}\Vert u\Vert_{ \mathcal{B}_T}.
\end{equation}
We also have
\begin{equation}\label{313}
\Vert u\Vert_{L^2(0,T;H^1(\RR))}\le C\Vert u\Vert_{ \mathcal{B}_T},
\end{equation}
where $C=\max \{ 1, \sqrt{T}\}.$
Using \eqref{312} and \eqref{313}, from \eqref{310} we get
\begin{equation}\label{314}
\Vert Mu -Mv\Vert_{\mathcal{B}_T}\le KT^{1/4}\left (
\Vert u\Vert_{\mathcal{B}_T}+\Vert v\Vert_{\mathcal{B}_T}
\right ) \Vert u-v\Vert_{\mathcal{B}_T} \qtq{for all}u, v\in {\mathcal{B}_T},
\end{equation}
with a suitable positive constant $K$.
Fix two constants $0<T\le\tau, R>0$ to be  chosen later, and take $u\in B_R(0)\subset {\mathcal{B}}_T.$
Then from \eqref{314} with $v\equiv 0$ we obtain
\begin{equation*}
\Vert Mu -M0\Vert_{\mathcal{B}_T}\le C T^{1/4}  \Vert u\Vert_{\mathcal{B}_T}.
\end{equation*}
Hence
\begin{equation*}
\Vert Mu\Vert_{\mathcal{B}_T}\le C T^{1/4}  \Vert u\Vert_{\mathcal{B}_T} +
\left\Vert S(t)u_0(0)-\int_0^tS(t-s)\lambda_0 u_0(s-\tau )\, ds \right\Vert_{\mathcal{B}_T}
\end{equation*}
and then,
using Proposition \ref{p32} with $f\equiv 0$ (recall that $t\le\tau $) we obtain that
\begin{equation*}
\Vert Mu\Vert_{\mathcal{B}_T}\le C\left (\Vert u_0(0)\Vert_{L^2(\RR)} +\Vert u\Vert_{L^2(-\tau,0;L^2(\RR))}+ T^{1/4}\Vert u\Vert^2_{\mathcal{B}_T}\right ).
\end{equation*}
It follows that $M$ maps $B_R(0)$ into itself if we choose $R=2C\left (\Vert u_0(0)\Vert_{L^2(\RR)} +\Vert u\Vert_{L^2(-\tau,0;L^2(\RR))}\right )$ and $T>0$ small enough.
Moreover, by \eqref{314} $M$ is a contraction if $T$ is sufficiently small. This proves the local well-posedness result for $0<T\le \tau$ small enough. Arguing as in the proof of Proposition \ref{p32} we obtain \eqref{39} for all $t\in [0,T].$
\end{proof}

\begin{proof}[Proof of Theorem \ref{t33}]
In order to prove that the solution is global we need to show that its norm remains bounded in the existence time interval.
For this purpose, we consider the functional ${\mathcal E}(\cdot)$ defined in \eqref{25}.
By differentiating ${\mathcal E}(t)$ we have
\begin{multline*}
\frac {d{\mathcal E}} {dt} (t)
=\int_\RR u(t)(u_{xx}(t)-\lambda_0 u(t)-\lambda u(t-\tau )+u(t)u_x(t)) dx+\frac 12\int_\RR \vert \lambda\vert u^2(t) dx
\\- \frac 12 e^{-\tau}\int_\RR \vert \lambda\vert u^2(t-\tau ) dx -\frac 12 \int_{t-\tau}^t \int_\RR e^{-(t-s)}\vert \lambda\vert u^2(x, s) dx ds.
\end{multline*}
Integrating by parts, using the Young inequality and recalling \eqref{15} and \eqref{27}, hence we obtain that
\begin{align*}
\frac {d{\mathcal E}} {dt} (t)
&\le -\int_\RR u_x^2(t) dx -\alpha_0 \int_\RR u^2(t) dx +\frac {e^\tau +1} 2 \int_\RR \vert \lambda(x)\vert  u^2 (t) dx\\
&\qquad\qquad -\frac 12 \int_{t-\tau}^t \int_\RR e^{-(t-s)}\vert \lambda\vert u^2(x, s) dx ds\\
&\le
-\int_\RR u_x^2(t) dx - (\alpha_0 -\alpha ) \int_\RR u^2(t) dx +
\int_\RR \beta (x) u^2 (t)\, dx.
\end{align*}
We can handle the third integral as in the proof of Theorem \ref{t23}, using \eqref{28}, showing that
$  \frac {d{\mathcal E}} {dt} (t)\le 0.$ This ensures that $\Vert u(t)\Vert_{L^2 (\RR)}$ remains bounded
for $t\in [0, T].$
From \eqref{39} we then deduce that $\Vert u\Vert_{\mathcal{B}_T}$ remains bounded  for $t\in [0, T].$
Therefore the local solution $u$ given by Proposition  \ref{p35} can  be extended on $[0,\tau].$
Finally, once we have a solution $u\in {\mathcal B}_\tau$ we can apply the step by step argument of Theorem  \ref{t22} proving the existence of a global mild solution.
\end{proof}

\begin{Theorem}\label{t36}
If  $\lambda_0 , \lambda\in L^\infty (\RR)$  satisfy \eqref{15}, \eqref{27} and \eqref{28}, then the problem \eqref{11} is exponentially stable. In particular, the solutions $u$  \eqref{11} satisfy the inequalities
\begin{equation}\label{315}
{\mathcal E}(t) \le C(u_0) e^{-\gamma t}
\end{equation}
with $\gamma$ and $C(u_0)$ as in \eqref{210} and \eqref{211}.
\end{Theorem}

\begin{proof}
The proof is analogous to that of Theorem \ref{t23}, by now using the Lyapunov functional \eqref{24}.
\end{proof}

\section{A more general model
\label{s4}}

Actually, we may consider a more general dissipative damping: the coefficient $\lambda_0$ in \eqref{12} and \eqref{11} may also change sign.

We assume, as in \cite{CCKR}, that there exist a number $\alpha >0$  and a function $\beta\in L^p(\RR)$ for some $1\le p< \infty$, such that
\begin{equation}\label{41}
\lambda_0 (x) \ge \alpha_0 -\beta_0 (x)  \qtq{for a.e.}x\in \RR ,
\end{equation}
where the function $\beta_0$ satisfies
\begin{equation}\label{42}
\Vert \beta_0\Vert_p<\Big (  \frac {\alpha_0}{c_p} \Big )^{1-\frac 1 {2p}}
\end{equation}
with the constant $c_p$ defined in \eqref{26}.

We can prove the following exponential stability result.

\begin{Theorem}\label{t41}
Let $\lambda , \lambda_0\in L^\infty (\RR),$ and assume that $\lambda_0$ satisfies \eqref{41} and \eqref{42}.
If  there exist a positive constant $\alpha$ and a function $\beta\in L^p(\RR),$ with the same $p$ as in \eqref{42}, such that the function $\lambda$ satisfies
\eqref{27}
with
\begin{equation}\label{43}
0\le\alpha < \alpha_0\qtq{and}
\Vert  \beta_0 +\beta \Vert_p<\Big (  \frac {\alpha_0-\alpha }{c_p} \Big )^{1-\frac 1 {2p}}
\end{equation}
where $c_p$ is defined in \eqref{26}, then the problem \eqref{12} is exponentially stable.
In particular, the solutions of \eqref{12} satisfy the estimates
\begin{equation}\label{44}
{\mathcal E}(t) \le C(u_0) e^{-\tilde\gamma t}
\end{equation}
with
\begin{equation}\label{45}
\tilde\gamma = \min \left\{2 \Big (
\alpha_0 -\alpha -\frac {2p-1}{2p} \Big (\frac 2 p \Big )^{\frac 1 {2p-1}}
\Vert \beta +\beta_0\Vert_p^{\frac {2p}{2p-1}}
\Big ), 1\right \}
\end{equation}
and
$C(u_0)$ is defined in \eqref{211}.
\end{Theorem}

\begin{proof}
Analogously to the proof of Theorem \ref{t23},
differentiating ${\mathcal E}(t),$
integrating by parts and using the Young inequality we obtain the following estimate:
\begin{align*}
\frac {d{\mathcal E}} {dt} (t)
&\le -\int_\RR u_x^2(t) dx -\alpha_0 \int_\RR u^2(t) dx +\int_\RR (\beta_0 (x)+\frac {e^\tau+1} 2 \vert \lambda (x)\vert ) u^2 (t) dx\\
&\qquad\qquad -\frac 12 \int_{t-\tau}^t\int_\RR e^{-(t-s)} \vert \lambda\vert u^2(x,s) dx ds\\
&\le
-\int_\RR u_x^2(t) dx - (\alpha_0 -\alpha ) \int_\RR u^2(t) dx +
\int_\RR (\beta_0 (x)+\beta(x) ) u^2 (t) dx\\
&\qquad\qquad -\frac 12 \int_{t-\tau}^t\int_\RR e^{-(t-s)} \vert \lambda\vert u^2(x,s) dx ds.
\end{align*}
Using the Hölder inequality hence we infer that
\begin{equation}\label{46}
\frac {d{\mathcal E}} {dt} (t)\le-\Vert u_x(t)\Vert_2^2 -(\alpha_0 -\alpha)\Vert u(t)\Vert_2^2
+\Vert \beta_0 +\beta\Vert_p \Vert u\Vert_{2q}^2 -\frac 12 \int_{t-\tau}^t\int_\RR e^{-(t-s)} \vert \lambda\vert u^2(x,s) dx ds,
\end{equation}
where $q=\frac p {p-1}.$
From \eqref{46} we deduce that
\begin{multline*}
\frac {d{\mathcal E}} {dt} (t)\le-\Vert u_x(t)\Vert_2^2 -(\alpha_0 -\alpha)\Vert u(t)\Vert_2^2+\Vert \beta +\beta_0\Vert_p  \Vert u(t)\Vert_{\infty}^{\frac 2 p}  \Vert u\Vert_2^{\frac 2 q}\\
-\frac 12 \int_{t-\tau}^t\int_\RR e^{-(t-s)} \vert \lambda\vert u^2(x,s) dx ds,\quad
\end{multline*}
and then, recalling \eqref{215}
and  using the Young inequality, we obtain for every fixed $\delta >0$ the inequality
\begin{multline*}
\frac {d{\mathcal E}} {dt} (t)
\le -\Vert u_x(t)\Vert_2^2 -(\alpha_0 -\alpha)\Vert u(t)\Vert_2^2
+\frac {
\Big ( \frac 1 {\delta } \Vert \beta_0 +
\beta\Vert_p\Vert u\Vert_2^{\frac {2p-1} p}  \Big )^\frac {2p}{2p-1}}{\frac {2p}{2p-1}} +
 \frac {
 \left ( \delta 2^{\frac 1 p}\Vert u_x\Vert_2^{\frac 1 p}      \right )^{2p}}{2p}\\
 -\frac 12 \int_{t-\tau}^t\int_\RR e^{-(t-s)} \vert \lambda\vert u^2(x,s) dx ds.
\end{multline*}
Now, taking $\delta$ as before,   such that $4\delta ^{2p}=2p,$ hence we infer that
\begin{multline*}
\frac {d{\mathcal E}} {dt} (t)\le - \Big (
\alpha -\alpha_0 -\frac {2p-1}{2p} \Big (\frac 2 p \Big )^{\frac 1 {2p-1}}
\Vert \beta_0 +\beta\Vert_p^{\frac {2p}{2p-1}}
\Big ) \Vert u(t)\Vert_2^2
-\frac 12 \int_{t-\tau}^t\int_\RR e^{-(t-s)} \vert \lambda\vert u^2(x,s) dx ds.
\end{multline*}
Thus, under the assumption \eqref{43} we have
\begin{equation*}
\frac {d{\mathcal E}} {dt} (t)\le - \tilde\gamma {\mathcal E}(t)
\end{equation*}
with
\begin{equation*}
\tilde\gamma= 2 \Big (
\alpha_0 -\alpha -\frac {2p-1}{2p} \Big (\frac 2 p \Big )^{\frac 1 {2p-1}}
\Vert \beta +\beta_0\Vert_p^{\frac {2p}{2p-1}}
\Big ).
\end{equation*}
This implies the exponential estimate \eqref{44}
with $C(u_0)$ as in \eqref{211}.
\end{proof}

In the same spirit, we can also extend the well-posedness and the stability result in the nonlinear setting.

\begin{Theorem}\label{t42}
Let $\lambda_0, \lambda\in L^\infty (\RR)$  satisfy \eqref{41}, \eqref{42},\eqref{27} and \eqref{43}.
Then for every $u_0\in C([-\tau,0]; L^2(\RR))$, the problem \eqref{11} has a unique global mild solution.
Moreover, the identity  \eqref{39} holds for all $t\ge 0$.
\end{Theorem}

\begin{Theorem}\label{t43}
Let  $\lambda_0, \lambda\in L^\infty (\RR)$  satisfy \eqref{41}, \eqref{42}, \eqref{27} and \eqref{43}.
Then, the problem \eqref{11} is exponentially stable.
In particular, the solutions of \eqref{12} satisfy the estimate
\begin{equation*}
{\mathcal E}(t) \le C(u_0) e^{-\tilde\gamma t}
\end{equation*}
with $\tilde\gamma$ as in \eqref{45}
and $C(u_0)$ as in \eqref{211}.
\end{Theorem}

\section*{Acknowledgments}
The first author was supported by the grant NSFC No. 11871348. The research of the second author was partially supported by
GNAMPA 2018 project ``Analisi e controllo di modelli differenziali non lineari'' (INdAM).
This work has been initiated during the first author's visit of the Department DISIM of Università di L'Aquila in December 2016.
He thanks the members of the department for their hospitality.

\end{document}